\documentclass[12pt]{amsart}
\usepackage{enumerate}
\usepackage{amsthm}
\usepackage{amsmath,amssymb,latexsym}
\makeatletter
\@namedef{subjclassname@2010}{%
  \textup{2010} Mathematics Subject Classification}
\makeatother
\theoremstyle{definition}
\newtheorem{theorem}{Theorem}[section]
\newtheorem{lemma}[theorem]{Lemma}
\newtheorem{proposition}[theorem]{Proposition}
\newtheorem{corollary}[theorem]{Corollary}
\theoremstyle{definition}
\theoremstyle{definition}
\newtheorem{definition}[theorem]{Definition}
\newtheorem{remark}[theorem]{Remark}
\newtheorem{example}[theorem]{Example}
\usepackage{indentfirst}

\begin{document}
\baselineskip=17pt
\title[]{Free actions of connected Lie groups of rank one on certain spaces}
\author[Hemant Kumar Singh, Jaspreet Kaur and Tej. B. Singh]{Hemant Kumar Singh, Jaspreet Kaur and Tej Bahadur Singh}
\address{{\bf Hemant Kumar Singh}, 
Department of Mathematics, Keshav Mahavidyalaya, H-4-5 zone, Sainik Vihar Pitampura, Delhi- 110034, India.}\email{
hksinghdu@gmail.com}
\address{{\bf Jaspreet Kaur},
Department of Mathematics, University of Delhi,
Delhi -- 110007, India.}
\email{jasp.maths@gmail.com}
\address{
{\bf Tej Bahadur Singh},
Department of Mathematics, University of Delhi,
Delhi -- 110007, India.}
 \email{tbsingh@maths.du.ac.in}

\date{}

\begin{abstract} 
 Let $G$ be a connected Lie group of rank one. In this paper the existence of free actions of group $G$ on spheres, real projective spaces and lens spaces has been studied. Most of the results have been obtained for finitistic spaces with cohomology ring isomorphic to the cohomology ring of these spaces. The cohomology algebra of the orbit space has been determined in the case of free $G$-action. The non-existence of $G$-equivariant maps from spheres to cohomology spheres has also been established.   
\end{abstract}
\subjclass[2010]{Primary 57S17; Secondary 55R20, 55M20 }

\keywords{Free action; finitistic space; Leray spectral sequence; mod $p$ cohomology algebra}

\maketitle

\section {Introduction}
The classification of finite groups that can act freely on complexes homotopic to sphere is an old and well studied problem in the theory of transformation groups. Popularly known as `topological spherical space form problem', this problem was stated by H. Hopf in the year 1925. But it attracted the interest of topologists only after J. Milnor's work \cite{Mil} in 1957. A survey by J.F. Davis and R.J. Milgram \cite{Dav} gives in detail the development related to this problem. This problem has been posed and studied for spaces other than sphere and groups other than finite groups as well. For instance B. Oliver \cite{Oli} showed that a compact Lie group has a free action on some product of spheres $\prod^{n}\mathbb S^k$ iff it has no subgroup isomorphic to $SO(3)$ (the group of rotations about the origin of three dimensional space $\mathbb R^3$). We, in this paper discuss the problem of existence of free actions of connected Lie groups of rank one on sphere, real projective spaces and lens spaces.\\ 
\indent{ \indent Real projective spaces and lens spaces are orbit spaces of a free linear action of a finite cyclic groups on spheres. An action $\theta$ of a topological group $G$ on a Hausdorff topological space $X$ is called \textit{free} if the \textit{isotropy subgroup} $G_x=\{g\in G|\theta(g,x)=x\}$ of $G$ at $x$ is trivial for each $x\in X$. The subset $G(x)=\{\theta(g,x)| g\in G\}$ of $X$ is called an \textit{orbit} and the set $X/G=\{G(x)|x\in X\}$ endowed with the quotient topology is called the \textit{orbit space}. The map (continuous function) sending $x\in X$ to $G(x)\in X/G$ is called the \textit {canonical projection} and will be denoted by $\pi:X\rightarrow X/G$. A real projective space and a lens space is thus described as follows:\\ 
\indent A \textit{$n$-dimensional real projective space} $\mathbb RP^n$ is the orbit space of the antipodal involution on an $n$-dimensional sphere $\mathbb S^n$ given by $$(x_1,x_2,\cdots, x_{n+1})\rightarrow (-x_1,-x_2,\cdots,-x_{n+1}).$$  Note that an involution on a space $X$ describes a continuous action of the group $\mathbb Z_2$ on $X$.  Now consider $(2n-1)$-dimensional sphere $\mathbb S^{2n-1}\subset \mathbb C^m$ as $\{z=(z_1,\cdots,z_n)|\left\|z\right\|=1\}$
Let $\epsilon =e^{2\pi\iota/p}$ be a primitive $p$th root of unity and let $q_1, q_2, \cdots , q_n$ be integers relatively prime to $p$. Then the map
$(z_1, . . . , z_n)\rightarrow (\epsilon^{q_1}z_1,\cdots,\epsilon^{q_n}z_n )$
defines a free action of the cyclic group $\mathbb Z_p$ on $\mathbb S^{2n-1}$. The orbit space of this action is denoted by $L^{2n-1}(p ;q_1, . . . , q_n)$ and is called a \textit{lens space}.We note that $L^{2n-1}(2 ;1, . . . , 1)$ is nothing but the $(2n-1)$-dimensional real projective space $\mathbb RP^{2n-1}$. For convenience, we at times write $L^{2n-1}(p ;q_1, . . . , q_n)$ as $L^{2n-1}(p,q)$. 

\indent It is well known that there are precisely three connected Lie groups of rank one, namely the circle group $\mathbb S^1$, the group of unit quaternions $\mathbb S^3$ and the group of rotations $SO(3)$. Now, by Lefschetz fixed point theorem it is well known that only finite group which can act freely on $\mathbb S^n$, $n$ even, is the cyclic group $\mathbb Z_2$. Thus, the circle group cannot act freely on $\mathbb S^n$, $n$ even. However, free actions of the group $\mathbb S^1$ on $\mathbb S^n$, $n$ odd, are possible. For example, the map $\theta:\mathbb S^1\times \mathbb S^{2n-1}\rightarrow \mathbb S^{2n-1}$ given by $\theta(z,(z_1,\cdots, z_m))=(zz_1,\cdots, zz_n)$ defines a free action of the group $\mathbb S^1$ on $(2n-1)$-sphere $\mathbb S^{2n-1}$. Now, free actions of $\mathbb S^1$ on lens spaces can be easily constructed by considering free actions of $\mathbb S^1$ on odd dimensional spheres and using the fact that if a topological group $G$ acts freely on a Hausdorff topological space $X$ and $N$ is a normal subgroup of $G$ then $G/N$ acts freely on $X/N$. For example, the action $\theta$ in the above example induces a free action of the topological group $\mathbb S^1/N \cong \mathbb S^1$ (topologically isomorphic) on the lens space $L^{2n-1}(p,1,\cdots, 1)=\mathbb S^{2n-1}/N$, $N=<\xi>$, $\xi=e^{2\pi \iota/p}$. In particular for $p=2$, we get a free action of $\mathbb S^1$ on an odd dimensional real projective space $\mathbb RP^{2n-1}$. Using Floyd-Euler characteristic formula \cite{Bre} it can be shown that the group $\mathbb Z_2$ cannot act freely on even dimensional real projective spaces. Thus we note that the existence of a free action of circle group $\mathbb S^1$ on $\mathbb RP^n$, $n$ even, is not possible. By $X\sim_p \mathbb S^n$, we mean that the cohomology ring $H^*(X;\mathbb Z_p)$ is isomorphic to the cohomology ring $H^*(\mathbb S^n;\mathbb Z_p)$. We will call such a space as \textit{mod $p$ cohomology sphere}. Similarly if $H^*(X;\mathbb Z_p)\cong H^*(\mathbb RP^n;\mathbb Z_p)$ (resp. $H^*(X;\mathbb Z_p)\cong H^*(L^{2n-1}(p ,q);\mathbb Z_p)$) then $X$ will be called \textit{mod $p$ cohomology real projective space}( resp.\textit{mod $p$ cohomology lens space}) and in short we will write $X\sim_p \mathbb RP^n$ (resp. $X\sim_p L^{2n-1}(p,q)$). Throughout the paper we will use $\check{\mbox{C}}$ech cohomology with $\mathbb Z_p$ coefficients, $p$ prime.  We will employ spectral sequence techniques to prove that the group $\mathbb S^3$ cannot act freely on a finitistic space $X$ which is either a mod $2$ real projective space $\mathbb RP^n$, $n\in \mathbb N$ or mod $p$ lens space $L^{2n-1}(p ;q_1, . . . , q_n)$, $p$ an odd prime. However we show that if $\mathbb S^3$ acts freely on finitistic space $X\sim_2\mathbb S^m$ then $m=4n-1$ for some $n\in\mathbb N$. In this case will will also determine the possible mod $2$ cohomology algebra of the orbit space $X/\mathbb S^3$. We remark that the orbit space of free $\mathbb S^1$ action on a finitistic space $X\sim_2 \mathbb RP^n$, $n$ odd, and a finitistic space $X\sim_p L^{2n-1}(p ,q) $, $p$ an odd prime, has been determined in \cite{Hemant}. Now it is already known that the group $SO(3)$ cannot act freely on sphere of any dimension. Using this fact and some already available results on lifting of group actions we are able to conclude that the group $SO(3)$ cannot act freely on lens spaces. We will discuss this in Section 4 of this paper. We will also discuss the existence of free action of the group $SO(3)$ on real projective spaces in this section.

\section {Preliminaries for proof of main theorems}

 In this paper we will consider only paracompact and Hausdorff topological spaces. We begin by recalling the definition of a `finitistic space'. These spaces were introduced by Swan \cite{swan} in the year 1960. Since then these spaces have become popular in the study of topological transformation group. \begin{definition} A topological space $X$ is said to be \textit{finitistic} if every open covering of $X$ has a finite dimensional open refinement, where the dimension of a covering is one less than the maximum number of members of the covering which intersect non-trivially.\end{definition} Two main classes of finitistic spaces are the compact spaces and the finite-dimensional spaces (spaces with finite covering dimension). The importance of finitistic spaces in the theory of transformation group is reflected in the following theorem.

\begin{theorem}\cite{SD} If $G$ is a compact Lie group acting continuously on a space $X$, then the space $X$ is finitistic if and only if the orbit space $X/G$ is finitistic. \end{theorem}
We recall that an action of a topological group $G$ on a space $X$ is called \textit{semi-free} if for each $x\in X$ the isotropy subgroup $G_x$ is either trivial or all of $G$. Clearly, a free group action is semi-free while the converse is certainly not true. We shall denote the subspace of fixed points ($x\in X$ is \textit{fixed point} if $G_x=G$) by $X^G=\{x\in X\;|\;gx=x\;\forall\;g\in G\}$.\\
\indent In \cite{Bre-pap} G.E. Breadon derived an exact sequence for sphere bundles with singularities called the `Smith-Gysin sequence' which is a generalization of the Gysin sequence and an analogue of the Smith sequence. But for our purpose, we note here only the particular case as follows. 

\begin{theorem} Let the group $G=\mathbb{S}^3$ acts semifreely on the space $X$. Then there is the following Smith-Gysin long exact sequence (with arbitrary coefficients): 
$$....\rightarrow H^{i}(X/G,X^G)\rightarrow H^i(X)\rightarrow H^{i-3}(X/G,X^G)\oplus H^i(X^G)\rightarrow H^{i+1}(X/G,X^G)\rightarrow ....$$\end{theorem} The homomorphism $\check{H}^{i-3}(X/G,X^G)\stackrel{\mu^*}{\rightarrow} H^{i+1}(X/G,X^G)$ which appears in the above exact sequence is the cup product with an element $\omega$, the generator, of $H^4(X/G,X^G).$ \\

\indent{We now recall the construction of Borel space and some results on spectral sequences. For details we refer \cite{Mc}, \cite{Gau} and \cite {Tom}. Let the group $G$ be a compact Lie group acting (not necessarily freely) on a finitistic space $X$. Let 
$$G \hookrightarrow E_G \longrightarrow B_G$$
 be the universal principal $G$-bundle. Consider the diagonal action of $G$ on $X \times E_G$. Let 
$$X_G=(X \times E_G) /G$$ 
be the orbit space of the diagonal action on $X \times E_G$. Then the projection $X \times E_G \to E_G$ is $G$-equivariant and gives a fibration 
$$X\hookrightarrow X_G \longrightarrow B_G$$ 
called the \textit{Borel fibration} and the space $X_G$ is called the \textit{Borel space}. If compact Lie group acts freely on a space $X$, then $X\rightarrow X/G$ is a principal $G$-bundle and one can take $h:X/G\rightarrow B_G$ a classifying map for the $G$-bundle $X\rightarrow X/G$. We will exploit the Leray- Serre spectral sequence associated to the Borel fibration $X \hookrightarrow X_G \longrightarrow B_G$.}
We also note that, 

\begin{enumerate}\item  For $G=\mathbb S^3$, the classifying space $B_G$ is the infinite dimensional quaternionic projective space, $\mathbb HP^\infty$, and 
$H^*(B_G; \mathbb Z_p)\cong \mathbb Z_p[x]$, deg $x=4$, 

\item $H^*(\mathbb S^n,\mathbb Z_2)\cong\mathbb Z_2[x]/<x^2>$, deg $x=n$,

\item $H^*(\mathbb RP^n,\mathbb Z_2)\cong \mathbb Z_2[x]/<x^{n+1}>$, deg $x=1$,  and 

\item For odd prime $p$, $H^*(L^{2n-1}(p,q),\mathbb Z_p)\cong \mathbb Z_p[x,z]/<x^2,z^n>$, deg $x=1$, deg $z=2$, and $\beta(x)=z$ where $\beta:H^1(L^{2n-1}(p,q),\mathbb Z_p)\rightarrow H^2(L^{2n-1}(p,q),\mathbb Z_p)$ is Bockstein homomorphism associated to the coefficient sequence $0\rightarrow \mathbb Z_p\rightarrow Z_{p^2}\rightarrow\mathbb Z_p\rightarrow 0$. \end{enumerate}

\begin{proposition}
Let $X\stackrel{i}{\hookrightarrow} X_G \stackrel{\pi}{\longrightarrow} B_G$ be the Borel fibration. Suppose that the system of local coefficients on $B_G$ is simple, then the edge homomorphisms
{\setlength\arraycolsep{35pt}
\begin{eqnarray}
\lefteqn{ H^k(B_G)=E_2^{k,0} \longrightarrow E_3^{k,0}\longrightarrow \cdots  }
                    \nonumber\\
& &   \longrightarrow E_k^{k,0} \longrightarrow E_{k+1}^{k,0}=E_{\infty}^{k,0}\subset H^k(X_G) \nonumber
\end{eqnarray}}
and  $$H^l(X_G) \longrightarrow E_{\infty}^{0,l}= E_{l+1}^{0,l} \subset E_{l}^{0,l} \subset \cdots \subset E_2^{0,l}= H^l(X)$$
are the homomorphisms $$\pi^*: H^k(B_G) \to H^k(X_G) ~ ~ ~ \textrm{and} ~ ~ ~ i^*: H^l(X_G)  \to H^l(X).$$
\end{proposition}

\begin{proposition}
Let $G$ be a compact Lie group acting freely on a finitistic space $X$. Then $$h^*: H^*(X/G) \stackrel{\cong}{\longrightarrow} H^*(X_G).$$
\end{proposition}

In fact $X/G$ and $X_G$ have the same homotopy type. Further, the $E_2$-term of the spectral sequence in this case is given by $E_2^{p,q}=H^p(B_G;(\mathcal H^q(X)))$, where $\mathcal H^q$ means a locally constant sheaf with stalk $\mathcal H^q(X)$ and group $G$, and it converges to $H^*(X_G)$. If the fundamental group $\pi_1(B_G)$ acts trivially on the cohomology of the fibre, $H^*(X)$, then the $E_2$-term of the spectral sequence takes the simple form $E_2^{p,q}=H^p(B_G)\otimes H^q(X)$.  \\
\indent{Let a group $G$ acts on spaces $X$ and $Y$. Then a continuous map $f:X\rightarrow Y$ is called \textit{$G$-equivariant} if for any $g\in G$ and $x\in X$, $f(g.x)=g.f(x)$. In \cite{Mattos} C. Biasi and D. Mattos proved the following Borsuk-Ulam type theorem, which deals with non-existence of $G$-equivariant maps.
\begin{theorem}Let $R$ be a P.I.D and $G$ be a compact Lie group acting freely on path connected spaces $X$ and $Y$. Suppose for some $m\in M$ $$H^q(X,R)=0 \;; 0<q<m$$ $$H^{m+1}(Y/G,R)=0.$$ If $\beta_m(X,R)<\beta_{m+1}(B_G,R)$, where $B_G$ is classifying space of $G$ and $\beta_i$ denotes the ith Betti number. Then there is no $G$-equivariant map $f:X\rightarrow Y$.\end{theorem}}

\section{ Main theorems}

In this section we consider the question of existence of free action of the group $\mathbb S^3$ on finitistic mod $2$ sphere, mod $2$ real projective space and mod $p$ lens space, $p$ an odd prime. We first prove a lemma analogous to Proposition 10.7 Chap. III\cite{Bre}.

\begin{lemma} Let $G=\mathbb{S}^3$ act freely on a finitistic space $X$ with $H^i(X,\mathbb Z_p)=0$, $\forall i>n$. Then $H^i(X/G,\mathbb Z_p)=0$ $\forall i> n$. \end{lemma}

\begin{proof} Using Smith-Gysin sequence (cf. Theorem 2.3), 
$$....\rightarrow H^{i}(X/G,\mathbb Z_p)\rightarrow H^i(X,\mathbb Z_p)\rightarrow H^{i-3}(X/G,\mathbb Z_p)\rightarrow H^{i+1}(X/G,\mathbb Z_p)\rightarrow ....$$
 we get that the map $\mu^*_i: H^{i-3}(X/G,\mathbb Z_p)\rightarrow H^{i+1}(X/G,\mathbb Z_p)$ which is cup product with the generator $\omega\in H^4(X/G,\mathbb Z_p)$, is an isomorphism for all $i>n$. Now since the finite dimensional invariant open coverings $\{\mathcal U\}$ of the orbit space $X/G$ are cofinal in the set of all open coverings of $X/G$ (cf. Theorem 2.2), we can write $\check{H}^*(X/G,\mathbb Z_p)= \varinjlim \check{H}^*(K(\mathcal U),\mathbb Z_p)$, where $K(\mathcal U)$ denotes the nerve of $\mathcal U$. Now, let $i_0>n$ and $\nu \in H^{i_0}(X/G,\mathbb Z_p)$ be arbitrary. We can find a finite dimensional invariant open covering $\mathcal V$ of $X/G$ and elements $\nu'\in H^{i_0}(K(\mathcal V),\mathbb Z_p)$ and $\omega' \in H^4(K(\mathcal V),\mathbb Z_p )$ such that $\rho(\nu')=\nu$ and $\rho (\omega ')=\omega$, where $\rho:H^{i_0}(K(\mathcal V),\mathbb Z_p)\rightarrow H^{i_0}(X/G,\mathbb Z_p)$ is the canonical map. Thus, for $k$, such that $4k+i_0>$ dim $\mathcal V$, we have $\omega'^k\nu'=0$ and this gives $(\mu^*_{i_0})^k(\nu)=\omega^k \nu=0$. The map $(\mu^*_{i_0})^k$ being a monomorphism gives, $\nu=0$. Hence, $ H^i(X/G,\mathbb Z_p)=0$ for all $i>n$. \end{proof}

Now, it is well known that the group $\mathbb S^3$ acts freely on homotopy $4n-1$ sphere, $n\in \mathbb N$ and the orbit space of any such free action is a space which is homotopy equivalent to a quaternionic projective space. We generalize this result to finitisitc mod $2$ cohomology spheres of dimension $4n-1$, $n\in \mathbb N$.

\begin{theorem} Let $G=\mathbb S^3$ act freely on a finitistic space $X\sim_2 \mathbb S^m$. Then $m=4n-1$, for some natural number $n$ and $H^*(X/G,\mathbb Z_2)\cong \mathbb Z_2[x]/<x^n>$, where deg $x=4$.\end{theorem}

\begin{proof}Let $G=\mathbb S^3$ act freely on $X$, then the Leray-Serre spectral sequence of the map $\pi:X_G\rightarrow B_G$ must not collapse at $E_2$-term. As $\pi_1(B_G)$ acts trivially on $H^*(X;\mathbb Z_2)$, the fibration $X\hookrightarrow X_G \longrightarrow B_G$ has a simple system of local coefficients on $B_G$. So the spectral sequence has the following form 
$$E_2^{k,l}\cong H^k(B_G;\mathbb Z_2)\otimes H^l(X,\mathbb Z_2)\cong \mathbb Z_2\otimes \mathbb Z_2\cong \mathbb Z_2\;\mbox{for}\; k=0,4,8,\cdots\;\text{and}\ \ l=0,m.$$
It is obvious that the differentials $d_r:E_r^{k,l}\rightarrow E_r^{k+r,l-r+1}$ are trivial for $r\leq m$. If $4$ does not divide $m+1$ then the differential $d_{m+1}:E_{m+1}^{k,l}\rightarrow E_{m+1}^{k+m+1,l-m}$ collapse at $E_2$-term, which contradicts our hypothesis using previous lemma. Thus, $m=4n-1$ for some $n\in \mathbb N$. 

Now let $a\in H^{4n-1}(X;\mathbb Z_2)$ be the generator of the cohomology ring $H^*(X;\mathbb Z_2)$ and $t\in H^4(B_G;\mathbb Z_2)$ be the generator of the cohomology ring $H^*(B_G;\mathbb Z_2)$. Then $1\otimes a$ is a generator of $E_2^{0,4n-1}$. If the differential
\begin{eqnarray*}d_{4n}:E^{0,4n-1}_{4n}\rightarrow E^{4n,0}_{4n}
\end{eqnarray*}
is the trivial homomorphism then multiplicative structure of the spectral sequence gives that $d_r=0$ for all $r$. Thus, the spectral sequence degenerates and there are fixed points of $G$ on $X$. Therefore, we must have $d_{4n}(1\otimes a)=t^n\otimes 1$. Consequently, the differentials $d_{4n}:E_{4n}^{k,4n-1}\rightarrow E_{4n}^{k+4n,0}$ are isomorphisms for all $k$. Thus, we obtain
\begin{eqnarray*}
E_{4n+1}^{k,l} =\begin{cases} \mathbb{Z}_2 & \text{for $k=4i(0\leq i\leq n-1)$ and $l=0$},\\
0 & \text{otherwise}.\end{cases}
\end{eqnarray*}
Now, the differentials $d_r:E_r^{k,l}\rightarrow E_r^{k+r,l-r+1}$ are the trivial homomorphisms for $m>4n$. Thus, $E_{4n+1}^{k,l}=E_\infty^{k,l}$ for all $k$ and $l$. It follows that $H^*(X_G,\mathbb Z_2)$ and Tot $E_\infty^{*,*}$ are same as the graded commutative algebra. The element $t\otimes 1\in E_2^{4,0}$ is a permanent cocycle and determines an element $x\in E_\infty^{4,0}=H^4(X_G;\mathbb Z_2)$. We have $\pi^*(t)=x$ and $x^n=0$. Therefore, the total complex Tot$E_\infty^{*,*}$ is the graded commutative algebra $$\mbox{Tot}\ E_\infty^{*,*}=\mathbb Z_2[x]/<x^n>,\;\;\mbox{where}\;\; \mbox{deg}x=4.$$ Since the action on $G$ is free, so we have $H^*(X_G,\mathbb Z_2)\cong H^*(X/G,\mathbb Z_2)$. Thus, $H^*(X/G,\mathbb Z_2)\cong \mathbb Z_2[x]/<x^n>\;\;\mbox{where},\;\; \mbox{deg}\;x=4$.
\end{proof}
\begin{remark}It is clear that the mod 2 cohomology algebra of the orbit space of a free action of $\mathbb{S}^3$ on a cohomology sphere $\mathbb{S}^{4n-1}$ is that of a cohomology quaternionic projective space $\mathbb{H}P^{n-1}$. The significance of the above result lies in the fact that the orbit spaces of different free $\mathbb{S}^3$ actions on a space may have different cohomology algebras. For example, one has two free $\mathbb{S}^3$ actions on $\mathbb{S}^7\times\mathbb{S}^{11}$ such that orbit spaces are $\mathbb{H}P^1\times\mathbb{S}^{11}$ and $\mathbb{S}^7\times \mathbb{H}P^2$.\end{remark}

Now, we define the mod 2 index of a finitistic space $X$ as follows:

\begin{definition} Let the group $G=\mathbb S^3$ act freely on a finitistic space $X\sim_2 \mathbb S^{4n-1}$, $n\in \mathbb N$. The mod $2$ index of this action is defined to be the largest integer $n$ such that $\alpha^n$ is non-zero, where $\alpha \in H^4(X/G,\mathbb Z_p)$ is the non-zero characteristic class of the $G$-bundle $G\hookrightarrow X\rightarrow X/G$. \end{definition}

\begin{theorem} Let $G=\mathbb{S}^3$ act freely on a finitistic space $X\sim_2 \mathbb S^{4n-1} $, $n>1$. Then the mod $2$ index of $X$ is $n-1$. \end{theorem}

\begin{proof} The Gysin sequence of the bundle $\mathbb S^3\hookrightarrow X \stackrel{\pi}{\rightarrow}X/G$ begins as $$0\rightarrow H^3(X/G,\mathbb Z_2)\stackrel{\pi^*}{\rightarrow} H^3(X,\mathbb Z_2)\rightarrow H^0(X/G,\mathbb Z_2)\stackrel{\mu^*}{\rightarrow}H^4(X/G,\mathbb Z_2)\rightarrow H^4(X,\mathbb Z_2)\rightarrow \cdots$$ The characteristic class of the bundle is defined to be the element $\mu^*(1)\in H^4(X/G,\mathbb Z_2)$, where $1$ is the unity of $H^0(X/G,\mathbb Z_2)$. In the Theorem 3.2,  $H^i(X/G,\mathbb Z_2)\cong Z_2$ for $i=4j$, $0\leq j \leq n-1$. Also, we have $H^i(X,\mathbb Z_2)\cong \mathbb Z_2$ for $i=0,4n-1$. Thus, $\mu^*: H^0(X/G,\mathbb Z_2)\rightarrow H^4(X/G,\mathbb Z_2)$ is an isomorphism. So, the characteristic class $\mu^*(1)=\alpha\in H^4(X/G,\mathbb Z_2)$ is non-zero. Thus by the Theorem 3.2, $\alpha^n=0$ but $\alpha ^{n-1}\neq 0$. Thus, mod $2$ index of $X$ is $n-1$.
\end{proof}

We now deduce Borsuk-Ulam type result from Theorems 2.6 and 3.2 for free $G=\mathbb{S}^3$ actions. Let $G$ act freely on $\mathbb S^{4k-1}$ and on a path connected finitistic space $Y\sim_2 \mathbb S^{4n-1}$. We note that if $k>n>1$ then by choosing $m=4n-1$ all the conditions of the Theorem 2.6 are satisfied and thus we are able to conclude the following:
\begin{corollary} Let $G=\mathbb{S}^3$ act freely on $\mathbb{S}^{4k-1}$ and on a path connected finitisitc space $Y\sim_2 \mathbb{S}^{4n-1}$. If $k-1>$ mod 2 index of $Y$ then there is no $G$-equivariant map $f:\mathbb{S}^{4k-1}\rightarrow Y$. \end{corollary}

For $G=\mathbb{S}^3$ actions on a mod 2 cohomology real projective space, we obtain

\begin{theorem} Let $X\sim_2\mathbb RP^n$, $n\in \mathbb N$ be a finitistic space. Then the group $G=\mathbb{S}^3$ cannot act freely on $X$.  \end{theorem}

\begin{proof} Suppose on the contrary, the group $G=\mathbb S^3$ acts freely on $X$. Then the Leray-Serre spectral sequence associated to the fibration $\pi:X_G\rightarrow B_G$ must not collapse at $E_2$-term. As $\pi_1(B_G)$ acts trivially on $H^*(X;\mathbb Z_2)$, so the spectral sequence has the following form \begin{eqnarray*}
E_2^{k,l} &\cong& H^k (B_G;\mathbb Z_2) \otimes H^l (X;\mathbb Z_2)\\[-3pt]
&\cong& \mathbb{Z}_2 \otimes \mathbb{Z}_2 \hskip.65in\ \text{for} \ \ k=0,4,8...\ \ \text{and} \ \ l=0,1,2,...,n\\[-3pt]
&\cong& \mathbb{Z}_2 \hskip1in \ \text{for} \ \ k=0,4,8...\ \ \text{and} \ \ l=0,1,2,...,n.
\end{eqnarray*}
Let $a\in H^1(X;\mathbb Z_2)$ be a generator of $H^*(X;\mathbb Z_2)$ and $t\in H^4(B_G,\mathbb Z_2)$ is a generator of $H^*(B_G,\mathbb Z_2)$. Since, by our hypothesis, $G$ has a free action on $X$ so some of the differentials $d_r:E_r^{k,l}\rightarrow E_r^{k+r,l-r+1}$ must be non-trivial. Clearly both $d_2$ and $d_3$ are trivial. If $d_4(1\otimes a^3)=t\otimes 1$, then $$t\otimes 1=d_4((1\otimes a)(1\otimes a^2))=d_4(1\otimes a).(1\otimes a^2)+(-1)^1(1\otimes a).d_4(1\otimes a^2)=0,$$ which is not possible. Therefore, the differential $d_4:E_4^{0,3}\rightarrow E_4^{4,0}$ must be the trivial homomorphism and the multiplication structure of the spectral sequence implies that $E_2=E_\infty$. Hence the spectral sequence degenerates and we get a contradiction using Lemma 3.1. Thus $G$ cannot act freely on $X$.  
\end{proof}

We prove analogous result for $G=\mathbb{S}^3$ actions on a mod $p$ cohomology lens space, where $p$ is an odd prime.

\begin{theorem} Let $X\sim_p L^{2m-1}(p,q)$ be finitistic space where $p$ an odd prime. Then $G=\mathbb{S}^3$ cannot act freely on $X$. \end{theorem}

\begin{proof} Suppose that $G=\mathbb{S}^3$ acts freely on $X$. Consider the Leray-Serre spectral sequence associated to the fibration $\pi:X_G\rightarrow B_G$. Since $\pi_1(B_G)$ acts trivially on $H^*(X;\mathbb{Z}_p)$ and so the fibration $X\overset{i}\hookrightarrow X_G\overset{\pi}{\rightarrow}B_G$ has a system of local coefficients on $B_G$. Therefore, the spectral sequence has
\begin{eqnarray*}
E_2^{k,l} &\cong& H^k (B_G;\mathbb Z_p) \otimes H^l (X;\mathbb Z_p)\cong \mathbb{Z}_p\ \text(for)\ k=0,4,8...\ \text{and} \ \ l=0,1,2,...,2m-1.
\end{eqnarray*}
Since $G=\mathbb{S}^3$ acts freely on $X$ then the spectral sequence does not collapse at $E_2$-term. Let $a\in H^1(X;\mathbb{Z}_p)$ and $b\in H^2(X;\mathbb{Z}_p)$ be generators of cohomology ring $H^*(X;\mathbb{Z}_p)$. Clearly, $d_2=0$ and $d_3=0$. If $d_4(1\otimes ab)=t\otimes 1$ where $t\in H^4(B_G;\mathbb{Z}_p)$ be the generator of the cohomology ring $H^*(B_G;\mathbb{Z}_p)$, then
 \begin{eqnarray*}
t\otimes 1=d_4((1\otimes b)(1\otimes a))=d_4(1\otimes b).(1\otimes a)+(-1)^2(1\otimes b).d_4(1\otimes a)=0,
\end{eqnarray*}
which is not possible. Therefore, the differential $d_4: E_4^{0,3} \to E_4^{4,0}$ must be trivial. Consequently, the spectral sequence degenerates and we get a contradiction using lemma 3.1. Thus $G$ cannot act freely on $X$. 
\end{proof}

\begin{remark}The above theorem also holds for the spaces whose mod $p$ cohomology is isomorphic to that of $\mathbb{S}^1\times\mathbb{C}P^{m-1}$.\end{remark}

\section{Remarks on action of $SO(3)$} One of the classical result in an attempt to solve the topological spherical space form problem is that if a finite group $G$ acts freely on some sphere $\mathbb S^n$, $n$ odd then the group $G$ must have periodic cohomology i.e $H^{n+1}(G,\mathbb Z)\cong \mathbb Z/|G|$ (Here H is Tate cohomology). This condition is equivalent to the condition that every abelian subgroup of $G$ must be cyclic. Also, by Lefschetz fixed point theorem it is well known that only finite group which can act freely on $\mathbb S^n$, $n$ even is the cyclic group $\mathbb Z_2$. Thus, the group $\mathbb Z_p\oplus \mathbb Z_p$ cannot act freely on sphere of any dimension. The above results can be generalized and we have, \\
\textbf{4.1.} The group $\mathbb Z_p\oplus \mathbb Z_p$ cannot act freely on any finite (dimensional) CW-complex which is homotopy equivalent to a $n$-sphere $\mathbb S^n$. \\
The group $SO(3)$ contains $\mathbb Z_2\oplus\mathbb Z_2$ as a subgroup, therefore by theorem above, we get\\
\textbf{4.2.} The group $SO(3)$ cannot act freely on any finite (dimensional) CW-complex homotopy equivalent to a sphere $\mathbb S^n$, $n\in \mathbb N$.\\

Now, we consider the question of existence of free action of $SO(3)$ on real projective spaces. The concept of \textit{free $p$-rank} of symmetry of a finite CW-complex $X$ for a prime $p$, introduced in \cite{Brow}, is defined as the rank of the largest elementary abelian $p$-group which acts freely on $X$. In \cite{Adem}, A. Adem et al. determined the free $2$-rank of symmetry of finite product of real projective spaces. \\ 
\textbf{4.3.} If $X=\mathbb (RP^n)^k$ and $\phi(X)$ denotes its free $2$-rank of symmetry then \[
\phi(X)=
     \begin{cases}
	0 & n\equiv 0,2 \;\mbox{mod}\;4\\k & n \equiv 1 \;\mbox{mod}\;4\\2k & n\equiv 3\;\mbox{mod}\;4
     \end{cases}
\]

From, this result we are able to conclude the following,
\begin{theorem} The group $SO(3)$ cannot act freely on $X=\mathbb RP^n$ when $n\equiv 0,1,2 \;\mbox{mod}\;4$\end{theorem}

We remark that the above theorem is infact valid for finite CW complexes homotopic to $\mathbb RP^n$, $n\equiv 0,1,2 \;\mbox{mod}\;4$. However, the group $SO(3)$ can act freely on spaces homotopic to $\mathbb RP^n$, when $n=4k-1$, $k\in \mathbb N$. 
\begin{example} Let $n\in \mathbb N$. \\ The function $\phi:\mathbb S^3\times \mathbb S^{4n-1}\rightarrow \mathbb S^{4n-1}$ as $$\theta (z,(w_1,w_2,w_3\cdots, w_{n}))=(zw_1,zw_2,zw_3,\cdots ,zw_{n})$$ defines free continuous action of the group $\mathbb S^3$ on $\mathbb S^{4n-1}$. The map $\phi$ induces a free action $\bar{\phi}$ of the group $\mathbb S^3/\mathbb Z_2\cong SO(3)$ on $\mathbb RP^{4n-1}$. 
\end{example}
Note that the orbit space $\mathbb RP^{4n-1}/SO(3)$ of $\bar{\phi}$ is the $(n-1)$-dimensional quaternionic projective space $\mathbb HP^{n-1}$.

 Using the classical result ( 4.2) and some results on lifting of group actions we will conclude shortly that the group $SO(3)$ cannot freely on lens space $L^{2n-1}(p,q)$, $p$ odd prime. Now, suppose a connected Lie group $G$ acts on a connected manifold (or more generally a CW complex) $M$. Then an action of group $G$ on $M$ can be lifted to an action of the universal cover group $\tilde{G}$ on any cover $N$ of $M$. But we may not always be able to lift the action of $G$ on $M$ to an action of the group $G$ itself on $N$. However, under certain conditions this becomes possible. One of the results in this direction is as follows.\\
 \textbf{4.4.}\cite{Got} is  Let $G$ be a compact Lie group and $M$ be a connected manifold (or CW-complex) having same homotopy type of a compact polyhedra, and  $\chi(M)\neq 0$ (Euler-characteristic). Then the $G$-action on $M$ lifts to a $G$-action on any cover of $M$.\\
The proof of the above result can also be found in \cite{Mont}. We note that in 4.4, if $G$-action on $M$ is free then the lifted action of $G$ on $N$ must also be free. Now, let the group $G=SO(3)$ act freely on a space $X=L^{2n-1}(p,q)$, $p$ an odd prime. Then by the above discussion, a free action of $G$ on $X$ lifts to a free action of $G$ on a universal covering space $\tilde{X}=S^{2n-1}$ of $X$. This contradicts 4.2. Thus, we conclude \begin{theorem} The group $G=SO(3)$ cannot act freely on $L^{2n-1}(p,q)$, $p$ an odd prime.\end{theorem} In fact by similar arguments we get that the group $SO(3)$ cannot act freely on a connected manifold which is homotopy equivalent to some compact polyhedra with Euler characteristic non-zero and a universal cover which is a finite (dimensional) CW-complex homotopic to a sphere. Fake lens spaces fall under this wide class of spaces, on which $SO(3)$ cannot act freely. A \textit {fake lens space} is the orbit space of a free action of finite cyclic group, $\mathbb Z_p$, $p\neq 2$, on an odd dimensional sphere. A fake lens space is a manifold homotopy equivalent to a lens space. To obtain fake lens spaces which are not homeomorphic to classical lens space more sophisticated techniques such as surgery theoretic methods are needed. One may see \cite{Wall} and \cite{Macko} for instance.

\bibliographystyle{amsplain}

\begin{thebibliography}{10}
	
\bibitem {Brow} A. Adem and W.  Browder, \textit{The Free Rank of Symmetry of $(\mathbb S^n)^k$}, Inventiones
Math. 92 (1988), 431–440.
\bibitem{Adem} A. Adem and E. Yalcin, \textit{On some examples of group actions and group extensions}, Journal of group theory, vol. 2 Issue 1 (2006), 69-79.
\bibitem {Mattos} C. Biasi and D. Mattos, \textit{A Borsuk-Ulam theorem for compact Lie group actions}, Bull. Braz. Math. Soc., New series, 37 (2006), 127-137.
\bibitem {Bre} G.E.  Bredon, \textit{Introduction to compact transformation groups}, Academic press (1972).
\bibitem{Bre-pap} G.E. Bredon, \textit{Cohomology fibre spaces, the Smith-Gysin sequence, and orientation in generalized manifolds}, Michigan Math. J. 10 (1963), 321-333.
\bibitem{Dav} J. F. Davis and R. J. Milgram,\textit{ A survey of the spherical space form problem}, Harwood (1984). 
\bibitem {Got} D. H. Gottlieb, \textit{A certain subgroup of the fundamental group}, Amer. J. Math., vol. 87 Issue 4 (1965), 840-856.
\bibitem {Macko} T. Macko and C. Wegner, \textit{On the classification of fake lens spaces}, Forum Mathematicum, Vol. 23 Issue 5 (2010), 1053-1091.
\bibitem {Mc} J. McCleary, \textit{A user's guide to spectral sequences}, Cambridge university press, II edition (2001).
\bibitem{Mil}J. Milnor, \textit{Groups which act on $\mathbb S^n$ without fixed points}, Amer. J. Math., 79 (1957), 623-630. 
\bibitem{Mont} J. Montaldi and J. Pablo-Ortega, \textit{Notes on lifting group actions}, www.manchester.ac.uk (mims, eprints), (2008).
\bibitem {Gau} G. Mukherjee et al., \textit{Transformation groups, symplectic torus actions and toric manifolds}, Hindustan Book Agency, (2005). 

\bibitem{Oli} R. Oliver, \textit{Free compact group actions on products of spheres},"Algebraic topology, Aarhus 1978", Lecture Notes in Math. 763, Springer (1979), 539-548.

\bibitem{SD} Satya Deo and H.S. Tripathi, \textit{Compact Lie group actions on finitistic spaces}, Topology vol. 21, No.4 (1982), 393-399.\bibitem {Hemant} H. K. Singh and Tej B. Singh, \textit{The cohomology of orbit spaces of certian free circle actions}, Proc. Indian Acad. Sci., vol 122 No. 1 (2012), 79-86
\bibitem{swan}  R. G. Swan, \textit{A new method in fixed point theory}, Comm. Math. Helv. 34 (1960), 1-16.
\bibitem{Tom} T. tom Dieck, \textit{Transformation groups}, de Gruyter Studies in Math., Walter de Gruyter, Berlin, 8 (1987).
\bibitem {Wall} C. T. C. Wall, \textit{Surgery on compact manifolds}, (Mathematical surveys and monographs) AMS, 2nd ed. 1999.





\end{thebibliography}

\end{document}